\documentclass[12 pt]{article}

\usepackage[utf8]{inputenc}
\usepackage{amsmath}
\usepackage{amsfonts}
\usepackage{amssymb}
\usepackage{mathtools}

\usepackage{cite}

\usepackage{graphicx}
\usepackage{caption}
\usepackage{listings}
\usepackage{verbatim}
\usepackage{wrapfig}
\usepackage{array}
\usepackage{subfigure}

\usepackage{tabularx}

\usepackage{caption}

\usepackage{amsthm}

\theoremstyle{lema}

\theoremstyle{proposition}
\newtheorem{proposition}{Proposition}[section]

\theoremstyle{theorem}
\newtheorem{theorem}{Theorem}[section]

\theoremstyle{theorem}
\newtheorem{remark}{Remark}[section]

\theoremstyle{corollary}
\newtheorem{corollary}{Corolarry}[section]

\theoremstyle{definition}
\newtheorem{definition}{Definition}[section]

\theoremstyle{example}
\newtheorem{example}{Example}[section]

\def\n{\mathbb N}

\providecommand{\keywords}[1]
{
	\small	
	\textbf{\textit{Keywords---}} #1
}

\providecommand{\msc}[1]
{
	\small	
	\textbf{\textit{Mathematics Subject Classification---}} #1
}

\title{Mappings contracting triangles}
\author{Ovidiu Popescu, Cristina Maria Păcurar\\  
\\ \small{Faculty of Mathematics and Computer Science,}\\ \small{Transilvania University of Bra\c sov, Bulevardul Eroilor 29, Bra\c sov}\\
\small{email: {ovidiu.popescu@unitbv.ro}, {cristina.pacurar@unitbv.ro}  } }
\date{}

\begin{document}
	
	\maketitle

	\begin{abstract}
		The aim of the current paper is to introduce a new class of contractive mappings, which are contracting (a feature of) triangles. We prove that maps contracting triangles are continuous and give the fixed point result for such mappings. We emphasize that our main theorem encompasses many functions, with significant applicability, for which the result holds, thereby representing a notable advancement in this research domain.
	\end{abstract}
	
	
	\msc{47H10, 47H09}

	\section{Introduction}
	
	\noindent 
	
	The research field of fixed point theorems began with Banach's pioneering work (see \cite{Banach}). The initial results proved by Banach were generalized in many directions, some providing results independent from the original result (see, for example, \cite{BoydWong,Chatterjea,Ciric,Kannan,HardyRogers,Meir-Keeler,Rakotch,Reich,Subrahmanyam,Suzuki,Wardowski} and more recent results \cite{Anjum,Berinde,Jleli,Pata,Petrusel,Popescu,Proinov}).
	
	The high interest of the research community towards this field is highlighted by the significant amount of papers covering fixed point results. 
	
	Recently, Petrov \cite{Petrov} introduced a new type of mappings $T:X \to X$, where $(X,d)$ is a metric space, which can be characterized as mappings contracting perimeters of triangles.
	
	\begin{definition}[Petrov \cite{Petrov}]
		Let $(X,d)$ be a metric space with $|X|\geq 3$. We shall say that $T:X\to X$ is a mapping contracting perimeters of triangles on $X$ if there exists $\alpha \in [0,1)$ such that the inequality 
		$$d(Tx,Ty) + d(Ty,Tz) + d(Tz,Tx) \leq \alpha [d(x,y)+d(y,z)+d(z,x)],$$
		holds for all three pairwise distinct points $x,y,x \in X$.
	\end{definition}

	Petrov showed that mappings contracting perimeters of triangles are continuous and proved the fixed point theorem for such mappings.
	
	\begin{theorem}[Petrov \cite{Petrov}]
		Let $(X,d)$, $|X|\geq 3$ be a complete metric space and let $T:X\to X$ be a mapping contracting perimeters of triangles on $X$. Then, $T$ has a fixed point if and only if $T$ does not possess periodic points of prime period $2$. The number of fixed points is at most $2$.
		\label{Petrov}
	\end{theorem}
	
	The results initially proven in \cite{Petrov} were extended to Kannan generalized mappings in \cite{Petrov-Kannan} and mappings contracting total pairwise distance in \cite{Petrov-2}.

	In this paper, we consider a new type of mappings $T:X\to X$ which can be characterized as mappings contracting (a feature) of triangles (not necessarily perimeters) and prove the fixed point theorem for such mappings. Although our proof is similar to classical proofs, we emphasize the fact that our results encompass a large family of three-valued functions, thus being a substantial advancement to the research field of fixed point theorems. Moreover, the examples of functions for which our results work have a significant applicability.
	
	\section{Mappings contracting triangles}
	
	\noindent
	
	Let $\Phi$ be the set of all functions $\varphi: [0,\infty)^3 \to [0,\infty)$ which satisfy the following properties:
	\begin{itemize}
		\item[$\Phi 1$)] $\varphi$ is symmetric:
		$$\varphi(a,b,c) = \varphi(a,c,b) = \varphi(b,a,c) = \varphi(b,c,a) = \varphi(c,a,b) = \varphi(c,b,a),$$
		for all $a,b,c \in [0,\infty)$;
		\item[$\Phi 2$)] $\varphi$ is continuous;
		\item[$\Phi 3$)] $\varphi$ is non-decreasing;
		\item[$\Phi 4$)] there exists $k >0$ such that $$k a \leq \varphi(a,b,c),$$
		for every $a,b,c \in [0,\infty)$;
		\item[$\Phi 5$)] $\varphi(a,b,c) = 0$ if and only if $a=b=c=0$.
	\end{itemize}

	\begin{remark}
		The following functions belong to $\Phi$:
		\begin{itemize}
			\item[1)] $\varphi(a,b,c) = a+b+c$;
			\item[2)] $\varphi(a,b,c) = \max\{a,b,c\}$;
			\item[3)] $\varphi(a,b,c) = \sqrt[p]{a^p+b^p+c^p}$, $p \in \n$, $p \geq 2$;
			\item[4)] $\varphi(a,b,c) = (\sqrt{a} + \sqrt{b} + \sqrt{c})^2$.
		\end{itemize}
	\end{remark}

	\begin{definition}
		Let $(X,d)$ be a metric space with $|X|\geq 3$. We shall say that $T:X\to X$ is a mapping contracting triangles on $X$ if there exists $\alpha \in [0,1)$ and $\varphi \in \Phi$ such that the inequality 
		\begin{equation}
			\varphi(d(Tx,Ty),d(Ty,Tz),d(Tz,Tx)) \leq \alpha \varphi(d(x,y),d(y,z),d(x,z)),
				\label{triangles}
		\end{equation}
		holds for all three pairwise distinct points $x,y,x \in X$.
	\end{definition}

	\begin{proposition}
		Mappings contracting triangles are continuous.
	\end{proposition}

	\begin{proof}
		If $x_0$ is an isolated point of $X$, then $T$ is continuous at $x_0$.
		
		Now, let $x_0$ be an accumulation point of $X$. Then, for every $\delta > 0$ there exists $y \in X$, $y\neq x_0$ such that $d(x_0,y) < \delta$. By (\ref{triangles}), we have for $x \in X$, $x \neq x_0$, $x\neq y$:
		\begin{equation*}
			\varphi(d(Tx_0,Tx),d(Tx_0,Ty),d(Tx,Ty)) \leq \alpha \varphi(d(x_0,x),d(x_0,y),d(x,y)).
		\end{equation*}
	
		If $d(x_0,x) < \delta$, since $\varphi$ is non-decreasing, we get 
		\begin{equation*}
			\varphi(d(Tx_0,Tx),d(Tx_0,Ty),d(Tx,Ty)) \leq \alpha \cdot \varphi(\delta, \delta, 2\delta) \leq \alpha \cdot \varphi(2\delta, 2\delta, 2\delta).
		\end{equation*}
	
		Since $\varphi \in \Phi$, by property $\Phi 4)$, there exists $k>0$ such that
		\begin{equation*}
			kd(Tx_0,Tx) \leq \alpha \cdot \varphi(2\delta, 2\delta, 2\delta).
		\end{equation*}
	
		Since $\varphi$ is continuous and $\varphi(0,0,0) = 0$, for every $\varepsilon >0$, there exists $\delta >0$ such that 
		$$\varphi(2\delta, 2\delta, 2\delta) < \varepsilon \dfrac{k}{\alpha}.$$
		
		Hence, we get $$kd(Tx_0, Tx) < \varepsilon k,$$
		by where $d(Tx_0,Tx) < \varepsilon$ which completes our proof.
	\end{proof}

	\begin{definition}
		Let $(X,d)$ be a metric space and $T:X\to X$. A point $x \in X$ is called a periodic point of period $n$ if $T^nx =x$. The least positive integer $n$ for which $T^n x=x$ is called the prime period of $x$.
	\end{definition}

	\begin{theorem}
		Let $(X,d)$, $|X|\geq 3$ be a complete metric space and let $T:X\to X$ be a mapping contracting triangles on $X$. Then, $T$ has a fixed point if and only if $T$ does not possess periodic points of prime period $2$. The number of fixed points is at most two.
		\label{Theorem}
	\end{theorem}

	\begin{proof}
		Suppose that $T$ has no periodic points of prime period $2$. 
		
		Let $x_0 \in X$, arbitrarily chosen, and $$x_{n+1}=Tx_n,$$ for every $n \geq 0$, the Picard iteration. Suppose that $x_i$ is not a fixed point of $T$ for every $i=0,1,\dots$ Then $x_i\neq x_{i+1} = Tx_i$ and $x_{i+2} = TTx_i\neq x_i$ for every $i=0,1,\dots$ Hence, $x_i$, $x_{i+1}$, $x_{i+2}$ are pairwise distinct. Now, set for every $i=0,1,\dots$:
		$$d_i = \varphi(d(x_i,x_{i+1}),d(x_{i+1},x_{i+2}),d(x_i,x_{i+2})),$$
		
		By (\ref{triangles}), setting $x=x_i$, $y=x_{i+1}$ and $z=x_{i+2}$, we get 
		
		\begin{equation*}
			\begin{gathered}
			\varphi(d(Tx_i,Tx_{i+1}),d(Tx_{i+1},Tx_{i+2}),d(Tx_{i},Tx_{i+2})) \leq\\ \leq \alpha \varphi(d(x_i,x_{i+1}),d(x_{i+1},x_{i+2}),d(x_i,x_{i+2})),
			\end{gathered}
		\end{equation*}
		for every $i=0,1,\dots$ Then,
		\begin{equation*}
			\varphi(d(x_{i+1},x_{i+2}),d(x_{i+2},x_{i+3}),d(x_{i+1},x_{i+3})) \leq \alpha \varphi(d(x_i,x_{i+1}),d(x_{i+1},x_{i+2}),d(x_i,x_{i+2})),
		\end{equation*}
		i.e.
		\begin{equation*}
			d_{i+1} \leq \alpha d_i,
		\end{equation*}
		for every $i=0,1,\dots$
		
		Since $\alpha<1$, the sequence $\{d_i\}$ is decreasing, i.e.
		\begin{equation*}
			d_0>d_1>d_2>\dots>d_n>\dots
		\end{equation*}
		
		Since $\varphi \in \Phi$ we have 
		\begin{equation*}
			\begin{aligned}
				kd(x_1,x_2) &\leq d_0,\\
				kd(x_2,x_3) &\leq d_1 \leq \alpha d_0,\\
				kd(x_3,x_4) &\leq d_2 \leq \alpha d_1 \leq \alpha^2 d_0,\\
				\dots\dots&\dots\dots \\
				kd(x_n,x_{n+1}) &\leq d_{n-1} \leq \alpha^{n-1} d_0,\\
				kd(x_{n+1},x_{n+2}) &\leq d_n \leq \alpha^n d_0,\\
				\dots\dots&\dots\dots 
			\end{aligned}
		\end{equation*}
	
		For $p \geq 1$, by the triangle inequality, we have 
		\begin{equation*}
			\begin{aligned}
				d(x_n,x_{n+p}) &\leq d(x_n,x_{n+1}) + d(x_{n+1},x_{n+2})+ \dots + d(x_{n+p-1},x_{n+p})\\
				&\leq\dfrac{(\alpha^{n-1}d_0+\alpha^{n}d_0+\dots+\alpha^{n+p-2}d_0)}{k}\\
				& = \dfrac{\alpha^{n-1}}{k} \cdot \dfrac{1-\alpha^p}{1-\alpha}d_0\\
				&\leq \dfrac{\alpha^{n-1}}{(1-\alpha)k}d_0.
			\end{aligned}
		\end{equation*}
	
		Hence, for every $\varepsilon > 0 $ there exists $n_0$ such that 
		$$ \dfrac{\alpha^{n-1}}{(1-\alpha)k}d_0 < \varepsilon,$$
		for every $n \geq n_0$.
		
		Then, we get $d(x_n,x_{n+p}) < \varepsilon$ for every $n \geq n_0$ and $p \geq 1$.
		
		Thus, $\{x_n\}$ is a Cauchy sequence. By completeness of $(X,d)$, this sequence has a limit $x^* \in X$.
		
		If there exists $n_1 \geq 0$ such that $x_n = x_{n_1}$ for every $n \geq n_1$ then $d_{n_1}=0$ which contradicts the monotonicity of $\{d_n\}$. Therefore, since $T$ does not possess periodic points of prime period $2$, there exists a subsequence $\{x_{n(k)}\}_{k \geq 0}$ such that $x_{n(k)} \neq x^*$ and $x_{n(k)+1} \neq x^*$ for every $k \geq 0$.
		
		Taking in (\ref{triangles}) $x=x_{n(k)}$, $y=x_{n(k)+1}$, $z=x^*$ we have 
		\begin{equation*}
			\begin{gathered}
				\varphi(d(Tx_{n(k)},Tx_{n(k)+1}),d(Tx_{n(k)+1},Tx^*),d(Tx_{n(k)},Tx^*)) \leq\\ \leq \alpha \varphi(d(x_{n(k)},x_{n(k)+1}),d(x_{n(k)+1},x^*),d(x_{n(k)},x^*)),
			\end{gathered}
		\end{equation*}
		by where 
		\begin{equation*}
			\begin{gathered}
				\varphi(d(x_{n(k)+1},x_{n(k)+2}),d(x_{n(k)+2},Tx^*),d(x_{n(k)+1},Tx^*)) \leq\\ \leq \alpha \varphi(d(x_{n(k)},x_{n(k)+1}),d(x_{n(k)+1},x^*),d(x_{n(k)},x^*)),
			\end{gathered}
		\end{equation*}
		
		Passing to the limit as $k \to \infty$, since $\varphi$ is continuous, we get
		\begin{equation*}
			\begin{gathered}
				\varphi(d(x^*,x^*),d(x^*,Tx^*),d(x^*,Tx^*)) \leq\\ \leq \alpha \varphi(d(x^*,x^*),d(x^*,x^*),d(x^*,x^*)),
			\end{gathered}
		\end{equation*}
		or 
		\begin{equation*}
				\varphi(0,d(x^*,Tx^*),d(x^*,Tx^*)) \leq \alpha \varphi(0,0,0) = 0.
		\end{equation*}
		
		Then, $\varphi(0,d(x^*,Tx^*),d(x^*,Tx^*)) = 0$, by where $d(x^*,Tx^*)=0$. Hence, $x^*$ is a fixed point of $T$.
		
		Now, we suppose that there exist three pairwise distinct fixed points of $T$, $x,y$ and $z$, i.e. $Tx=x$, $Ty=y$ and $Tz=z$.
		
		Then, by (\ref{triangles}) we have  
		\begin{equation*}
			\varphi(d(x,y),d(y,z),d(x,z)) \leq \alpha \varphi(d(x,y),d(y,z),d(x,z)),
		\end{equation*}
		by where $$\varphi(d(x,y),d(y,z),d(x,z)) = 0$$ which implies, via $\Phi 5)$, that $d(x,y) = d(y,z) = d(x,z) = 0$, which is a contradiction.
		
		Conversely, let $T$ have a fixed point $x^*$ and $x$ a periodic point of prime period $2$. 
		
		Then,
		\begin{equation*}
			\varphi(d(Tx,T^2x),d(T^2x,Tx^*),d(Tx^*,Tx)) \leq \alpha \varphi(d(x,Tx),d(Tx,x^*),d(x,x^*)).
		\end{equation*}
	
		Hence,
		\begin{equation*}
			\varphi(d(x,Tx),d(x,x^*),d(x^*,Tx)) \leq \alpha \varphi(d(x,Tx),d(Tx,x^*),d(x,x^*)).
		\end{equation*}
		
		Since $\varphi$ is symmetric , we get $\varphi(d(x,Tx),d(x,x^*),d(x^*,Tx))=0$, i.e. $d(x,Tx) =0$ which is a contradiction.
	\end{proof}

	\begin{remark}
		In the case when there exists some iteration $x_{n+1} = Tx_n$, $n \geq 1$ such that $x_n \to x*$, $Tx^*=x^*$ and $x_{n(k)} \neq x^*$, for $k \geq 1$, then $T$ has a unique fixed point.
		
		Indeed, suppose that $T$ has another fixed point $x^{**}\neq x^*$. Obviously, $x_{n(k)} \neq x^{**}$ for $k=1,2,\dots$ Hence, setting $x=x_{n(k)}$, $y=x^*$ and $z=x^{**}$ in (\ref{triangles}), we get for every $k \geq 1$ 
		\begin{equation*}
			\begin{gathered}
				\varphi(d(Tx_{n(k)},Tx^*),d(Tx^*,Tx^{**}),d(Tx_{n(k)},Tx^{**})) \leq\\ \leq \alpha \varphi(d(x_{n(k)},x^*),d(x^*,x^{**}),d(x_{n(k)},x^{**})),
			\end{gathered}
		\end{equation*}
		by where 
		\begin{equation*}
			\begin{gathered}
				\varphi(d(x_{n(k)+1},x^*),d(x^*,x^{**}),d(x_{n(k)+1},x^{**})) \leq\\ \leq \alpha \varphi(d(x_{n(k)},x^*),d(x^*,x^{**}),d(x_{n(k)},x^{**})).
			\end{gathered}
		\end{equation*}
	
		Since $\varphi$ is continuous, taking the limit as $k \to \infty$, we obtain
		\begin{equation*}
			\begin{gathered}
				\varphi(d(x^*,x^*),d(x^*,x^{**}),d(x^*,x^{**})) \leq \alpha \varphi(d(x^*,x^*),d(x^*,x^{**}),d(x^*,x^{**})).
			\end{gathered}
		\end{equation*}
		
		Hence, $\varphi(d(x^*,x^*),d(x^*,x^{**}),d(x^*,x^{**}))=0$. Since $\varphi \in \Phi$ we get $d(x^*,x^{**}) = 0$, which is a contradiction 
	\end{remark}

	\begin{remark}
		Taking $\varphi(a,b,c) = a+b+c$ in Theorem \ref{Theorem} we obtain Theorem \ref{Petrov}.
	\end{remark}

	\begin{corollary}
		Let $(X,d)$, $|X|\geq 3$ be a complete metric space and let $T:X\to X$ a mapping such that 
		\begin{equation*}
			\max\{d(Tx,Ty),d(Ty,Tz),d(Tz,Tx)\} \leq \alpha \varphi\max\{d(x,y),d(y,z),d(x,z)\},
		\end{equation*}
		for every pairwise distinct points $x,y,x \in X$, where $\alpha \in [0,1)$. Then, $T$ has a fixed point if and only if $T$ does not possess periodic points of prime period $2$. The number of fixed points is at most two.
		\label{2.7}
	\end{corollary}
	\begin{proof}
		We take $\varphi(a,b,c) = \max\{a,b,c\}$ in Theorem \ref{Theorem}.
	\end{proof}

	\begin{corollary}
		Let $(X,d)$, $|X|\geq 3$ be a complete metric space and let $T:X\to X$ a mapping such that 
		\begin{equation*}
			d^2(Tx,Ty)+d^2(Ty,Tz)+d^2(Tz,Tx)\} \leq \beta (d^2(x,y)+d^2(y,z)+d^2(x,z)),
		\end{equation*}
		for every pairwise distinct points $x,y,x \in X$, where $\beta \in [0,1)$. Then, $T$ has a fixed point if and only if $T$ does not possess periodic points of prime period $2$. The number of fixed points is at most two.
	\end{corollary}
	\begin{proof}
		We take $\varphi(a,b,c) = \sqrt{a^2+b^2+c^2}$ and $\alpha=\sqrt{\beta}$ in Theorem \ref{Theorem}.
	\end{proof}

	\begin{example}
		Let $X=\{A,B,C,D,E\}$, and as in Figure \ref{Ex1}, let
		$$d(A,B) = d(A,D) = d(A,E) =d(B,C) = d(B,E) = 4,$$
		$$d(C,E) = d(D,E) = 3$$
		$$d(A,C) = d(B,D) = d(C,D) = 2,$$
		and let $T: X \to X$ such that $TA=C$, $TB=D$, $TC=TD=TE=E$. 
		
		\begin{figure}[h!]
			\centering
			\includegraphics[width=8cm]{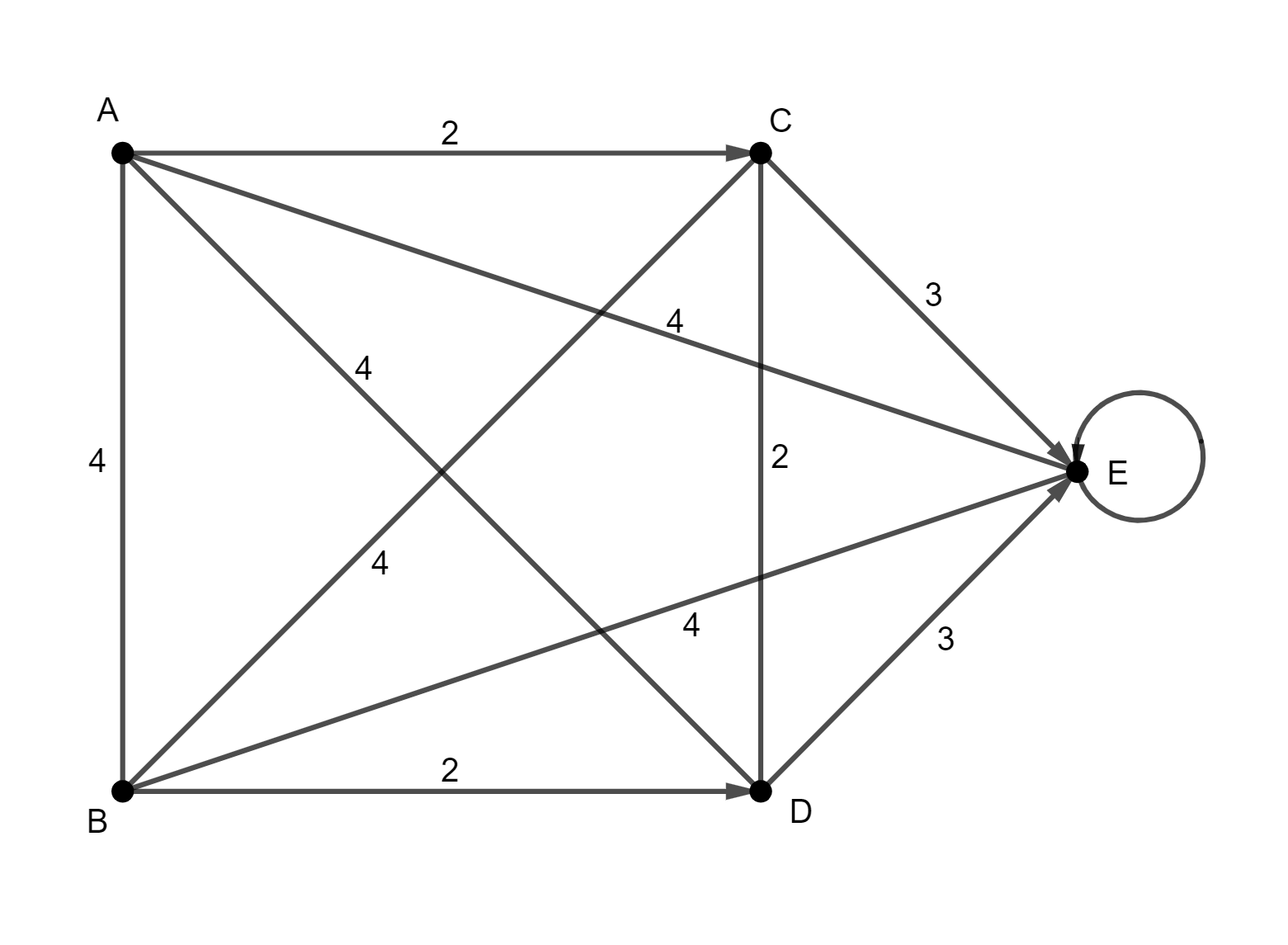}
			\caption{Example 2.1}
			\label{Ex1}
		\end{figure}
		Since $$d(TA,TB)+d(TB,TC)+d(TA,TC) = 10 = d(A,B) + d(B,C) + d(A,C),$$
		$T$ is not a mapping contracting perimeters of triangles, so Theorem \ref{Petrov} does not work. 
		
		However, we have
		\begin{equation*}
			\begin{gathered}
			M(TA,TB,TC) := \max \{d(TA,TB),d(TB,TC),d(TA,TC)\} = 3 \leq \\
			\leq  \alpha \max\{d(A,B),d(B,C),d(A,C)\} = \alpha M(A,B,C) = 4 \alpha,
		\end{gathered}
		\end{equation*}
		for $\alpha \geq \frac34$.
	
		Also, 
		\begin{equation*}
			\begin{aligned}
				M(TA,TB,TD) = 3 & \text{    and  } & M(A,B,D) = 4\\
				M(TA,TD,TE) = 3 & \text{    and  } & M(A,D,E) = 4\\
				M(TA,TC,TD) = 3 & \text{    and  } & M(A,C,D) = 4\\
				M(TA,TC,TE) = 3 & \text{    and  } & M(A,C,E) = 4\\
				M(TA,TD,TE) = 3 & \text{    and  } & M(A,D,E) = 4\\
				M(TB,TC,TD) = 3 & \text{    and  } & M(B,C,D) = 4\\
				M(TB,TC,TE) = 3 & \text{    and  } & M(B,C,E) = 4\\
				M(TB,TD,TE) = 3 & \text{    and  } & M(B,D,E) = 4\\
				M(TC,TD,TE) = 0 & \text{    and  } & M(C,D,E) = 3.\\
			\end{aligned}
		\end{equation*}
		Therefore, 
		$$M(Tx,Ty,Tz) \leq \dfrac34 M(x,y,z),$$
		for all three pairwise distinct points $x,y,z \in X$. Hence, we can apply Corollary \ref{2.7}. 
		
		We note that $T$ is not a contraction ($d(TA,TC) = 3 > 2 = d(A,C)$).	
	\end{example}

	\begin{example}
		Let $X=\{A,B,C,D\}$, and as in Figure \ref{Ex2} let
		$$d(A,B) = d(A,C) = d(A,D) = d(B,C) = d(B,D) = 2,$$
		$$d(C,D) = 3,$$
		and let $T: X \to X$ such that $TA=TC=C$, $TB=TD=D$.
		
		\begin{figure}[h!]
			\centering
			\includegraphics[width=5cm]{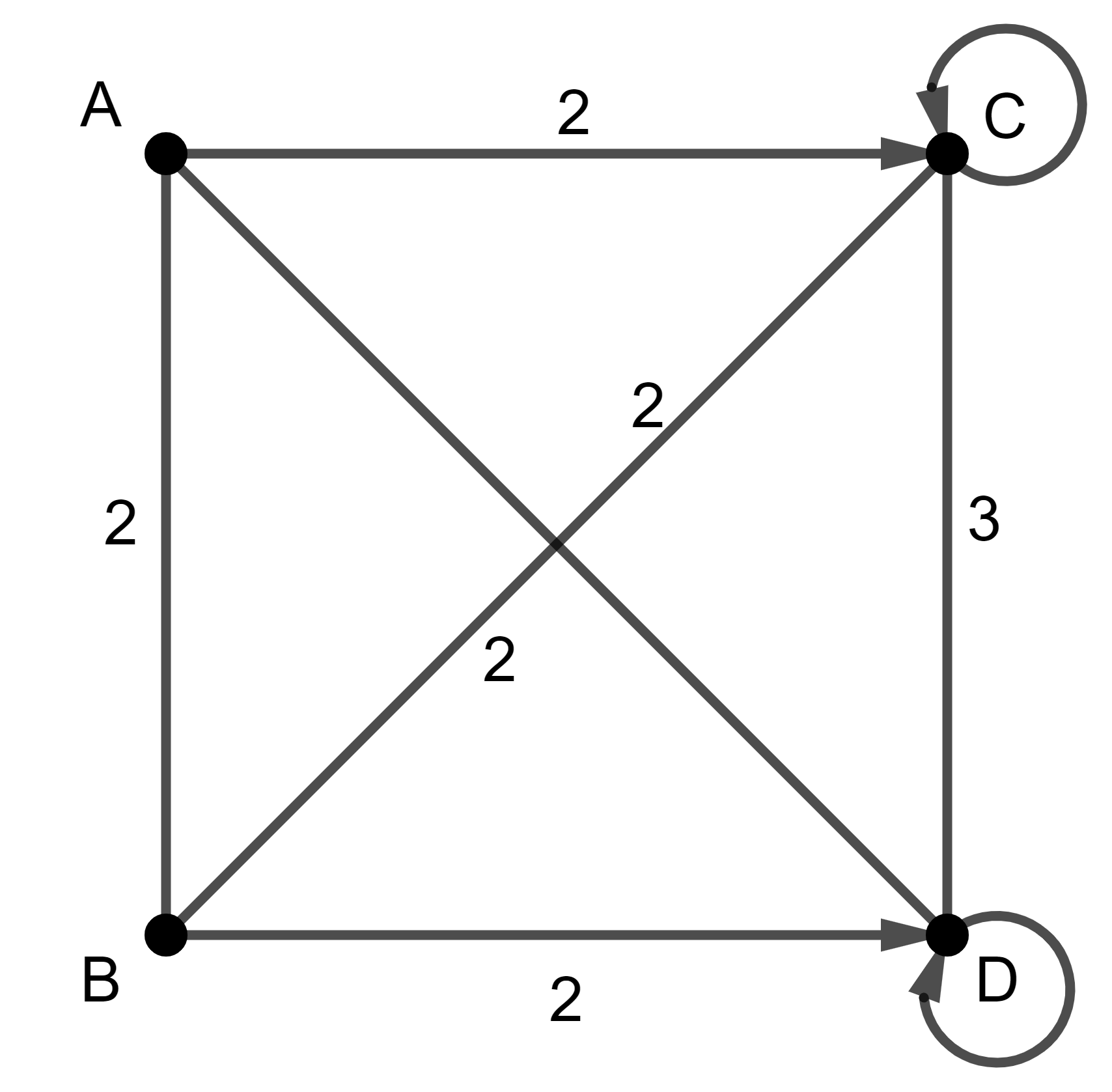}
			\caption{Example  2.2}
			\label{Ex2}
		\end{figure}
		
		Since  $$d(TA,TB)+d(TB,TC)+d(TA,TC) = 6 = d(A,B) + d(B,C) + d(A,C),$$
		$T$ does not contract perimeters of triangles. Moreover, $M(TA,TB,TC) = 3$ and $M(A,B,C) = 2$, so Corollary \ref{2.7} does not work.
		
		However, 
		\begin{equation*}
			N(TA,TB,TC) :=\left(\sqrt{d(TA,TB)}+\sqrt{d(TB,TC)}+\sqrt{d(TA,TC)}\right)^2 = (2\sqrt{3})^2 = 12
		\end{equation*}
		and
		\begin{equation*}
			N(A,B,C) = \left(\sqrt{d(A,B)}+\sqrt{d(B,C)}+\sqrt{d(A,C)}\right)^2 = (3\sqrt{2})^2 = 18,
		\end{equation*}
		so $N(TA,TB,TC) \leq \frac23 N(A,B,C)$.
		
		Also, 
		\begin{equation*}
			\begin{gathered}
				N(TA,TB,TD) = 12 \leq \dfrac23 \cdot 18 = \dfrac23 N(A,B,D)\\
				N(TA,TC,TD) = 12 \leq \dfrac23 \cdot (11+4\sqrt{6}) = \dfrac23 N(A,C,D)\\
				N(TB,TC,TD) = 12 \leq \dfrac23 \cdot (11+4\sqrt{6}) = \dfrac23 N(B,C,D).\\
			\end{gathered}
		\end{equation*}
	
		Therefore, we can apply Theorem \ref{Theorem} with $\varphi(a,b,c) = (\sqrt{a}+\sqrt{b}+\sqrt{c})^2$ and $\alpha = \dfrac23$.
	\end{example}

	\medskip
	\vspace{1.2ex}
	
\end{document}